\documentclass[11pt,leqno]{amsart}
\usepackage{amssymb,amsmath,amsthm,amsfonts}
\usepackage{graphicx}

\DeclareMathOperator{\Real}{Re} \DeclareMathOperator{\Imag}{Im}

\DeclareMathOperator{\diam}{diam}
\DeclareMathOperator{\dist}{dist}
\newcommand{\field}[1]{\mathbb{#1}}
\newcommand{\N}{\field{N}}                      
\newcommand{\R}{\field{R}}                      
\newcommand{\C}{\field{C}}                      

\newcommand{\eps}{\epsilon}

\newcommand{\loc}{{\scriptstyle{loc}}}

\newcommand{\RH}{{\scriptstyle{RH}}}

\newcommand{\Sob}{{\scriptstyle{Sob}}}

\newcommand{\bA}{{\mathbf A}}

\newcommand{\cC}{{\mathcal C}}

\newcommand{\cF}{{\mathcal F}}

\newcommand{\cU}{{\mathcal U}}

\newcommand{\bv}{{\mathbf v}}

\newcommand{\cW}{{\mathcal W}}

\newcommand{\ovdimB}{{\overline{\dim_B}}}

\newcommand{\bi}{{\mathbf i}}

\def\Barint_#1{\mathchoice
          {\mathop{\vrule width 6pt height 3 pt depth -2.5pt
                  \kern -8pt \intop}\nolimits_{#1}}%
          {\mathop{\vrule width 5pt height 3 pt depth -2.6pt
                  \kern -6pt \intop}\nolimits_{#1}}%
          {\mathop{\vrule width 5pt height 3 pt depth -2.6pt
                  \kern -6pt \intop}\nolimits_{#1}}%
          {\mathop{\vrule width 5pt height 3 pt depth -2.6pt
                  \kern -6pt \intop}\nolimits_{#1}}}

\theoremstyle{plain}
\newtheorem{theorem}{Theorem}
\newtheorem{corollary}[theorem]{Corollary}
\newtheorem{lemma}[theorem]{Lemma}
\newtheorem{proposition}[theorem]{Proposition}

\theoremstyle{definition}
\newtheorem{definition}[theorem]{Definition}
\newtheorem{example}[theorem]{Example}
\newtheorem{remark}[theorem]{Remark}

\newtheorem{question}[theorem]{Question}
\newtheorem{conjecture}[theorem]{Conjecture}

\numberwithin{theorem}{section} \numberwithin{equation}{section}

\title[QC distortion of the Assouad spectrum and polynomial spirals]{Quasiconformal distortion of the Assouad spectrum and classification of polynomial spirals}
\date{\today}
\author{Efstathios K. Chrontsios Garitsis}
\address{Department of Mathematics \\ University of Illinois at Urbana-Champaign \\ 1409 West Green Street \\ Urbana, IL 61801}
\email{echronts@gmail.com}
\author{Jeremy T. Tyson}
\address{Department of Mathematics \\ University of Illinois at Urbana-Champaign \\ 1409 West Green Street \\ Urbana, IL 61801 and National Science Foundation \\ 2415 Eisenhower Avenue \\ Alexandria, VA 22314}
\email{tyson@illinois.edu}

\begin{document}
\maketitle

\begin{center}
\it Dedicated to Frederick W. Gehring (1925--2012) and Jussi V\"ais\"al\"a
\end{center}

\begin{abstract}
We investigate the distortion of Assouad dimension and the Assouad spectrum under Euclidean quasiconformal maps. Our results complement existing conclusions for Hausdorff and box-counting dimension due to Gehring--V\"ais\"al\"a\ and others. As an application, we classify polynomial spirals $S_a:=\{x^{-a}e^{\bi x}:x>0\}$ up to quasiconformal equivalence, up to the level of the dilatation. Specifically, for $a>b>0$ we show that there exists a quasiconformal map $f$ of $\C$ with dilatation $K_f$ and $f(S_a)=S_b$ if and only if $K_f \ge \tfrac{a}{b}$.
\end{abstract}

\section{Introduction}

Quasiconformal mappings can distort dimensions of sets. Gehring and V\"ais\"al\"a \cite{GV} gave dilatation-dependent bounds for Euclidean quasiconformal distortion of Hausdorff dimension. The precise formulation of their bounds involves the sharp exponent of higher Sobolev integrability for an $n$-dimensional quasiconformal map, whose precise value remains conjectural in dimension at least three. In dimension two, explicit, sharp results follow from Astala's theorem \cite{Astala}.

A dilatation-independent study of quasiconformal dimension distortion was initiated in the late 1990s, see e.g.\ \cite{bis:increase}, \cite{tys:hausdorff}, \cite{kov:confdim}, \cite{tw:gasket}. While the results of Gehring and V\"ais\"al\"a concern the distortion of dimensions of arbitrary subsets by a fixed quasiconformal mapping, these later results concern distortion of dimension for a fixed subset of $\R^n$ by arbitrary quasiconformal maps.

A more recent line of research (see, for instance, \cite{bmt:frequency}, \cite{bhw:quasiconformal}, or \cite{btw:parallel-lines}) addresses the question of simultaneous distortion of dimensions of large families of parallel subspaces, or generic elements in other parameterized families of subsets.

In all of the preceding discussion, the concept of dimension under consideration is the Hausdorff dimension. Hausdorff dimension is one of the most well-studied metric notions of dimension, and numerous tools exist for its computation and estimation. Some of the preceding theory extends to other notions of dimension, such as box-counting or packing dimension. For example, the distortion bounds in Gehring and V\"ais\"al\"a's original paper hold also for both box-counting dimension and packing dimension, as they rely only on higher Sobolev regularity. See Kaufman, \cite{Kaufman}, for a discussion of the distortion of Hausdorff and box-counting dimension by supercritical Sobolev maps.

In this paper we establish dilatation-dependent estimates for the distortion of Assouad dimension and the Assouad spectrum by quasiconformal maps. This paper can be seen as a companion to \cite{tys:assouad}, which considered conformal Assouad dimension of sets and metric spaces. We improve and sharpen some dilatation-dependent estimates from \cite{tys:assouad}, and we initiate a study of quasiconformal distortion of the recently defined Assouad spectrum.

As an application, we provide a precise classification of planar polynomial spirals up to quasiconformal equivalence which is sharp on the level of dilatation. This result indicates the relevance of the Assouad spectrum for classification problems in geometric mapping theory.

For a strictly decreasing function $\phi:(1,\infty) \to (0,1)$ with $\phi(x) \stackrel{x\to\infty}{\longrightarrow} 0$, define the {\it $\phi$-spiral}
$$
S(\phi) := \{ \phi(x) e^{\bi x} \in \C : x>0 \}.
$$
When $\phi(x) = e^{-cx}$, $c>0$, we have the {\it logarithmic spiral}. Here we consider polynomial spirals. For $a>0$ set $S_a := S(\phi_a)$ where
$$
\phi_a(x) = x^{-a}.
$$

The following result is an application of our main theorem.

\begin{theorem}\label{th:spirals-classification}
For $a>b>0$, there exists a quasiconformal map $f:\C \to \C$ with $f(S_a) = S_b$ if and only if $K_f \ge \tfrac{a}{b}$.
\end{theorem}

One direction is trivial. The radial stretch map $f(z) = |z|^{1/K-1}z$ is $K$-quasiconformal, and maps $S_a$ to $S_b$ with $b=a/K$. The content of the theorem is the other implication, namely, if $K<\tfrac{a}{b}$, then no $K$-quasiconformal map of $\C$ satisfies $f(S_a) = S_b$. Distinguishing sets up to quasiconformal equivalence, particularly at the level of the dilatation, is in general a hard problem. One approach is to use dimension distortion bounds. Theorem \ref{th:spirals-classification} is the first application of the Assouad spectrum to such a question, and motivates the study of these notions of dimension in connection with geometric mapping theory.

We begin the discussion of quasiconformal dimension distortion estimates by recalling the celebrated estimates of Gehring and V\"ais\"al\"a from \cite{GV}. Let $f:\Omega \to \Omega'$ be a $K$-quasiconformal mapping between domains in Euclidean space $\R^n$, $n \ge 2$. Here and throughout this paper we adopt the analytic definition for quasiconformality, expressed in terms of the outer dilatation $K_O(f)$ of $f$, and we denote by $\Barint_E w$ the average of a locally integrable function $w:\R^n \to \R$ over a set $E$ of positive measure, namely, $\Barint_E w = |E|^{-1} \int_E w$. See subsection \ref{subsec:qc-mappings} for details. According to Gehring's well-known higher integrability theorem \cite[Theorem 1]{Gehring73}, $f$ lies in a local Sobolev space $W^{1,p}_\loc(\Omega:\R^n)$ for some $p>n$. In particular, the differential of $f$ satisfies a reverse H\"older inequality, to wit,
\begin{equation}\label{eq:pnK0}
\left( \Barint_Q |Df|^p \right)^{1/p} \le C \left( \Barint_{Q} |Df|^n \right)^{1/n} 
 \end{equation}
for all cubes $Q \subset \Omega$ with $\diam(Q) < \dist(Q,\partial\Omega)$ and $\diam(f(Q)) < \dist(f(Q),\partial\Omega')$. Here $C>0$ denotes a constant which depends only on the dimension $n$ and the dilatation $K$ of the mapping $f$.
 
For the purposes of this paper, we will use a weaker version of \eqref{eq:pnK0}. Namely, throughout the paper we will refer to the following reverse H\"older inequality:
\begin{equation}\label{eq:pnK}
		\left( \Barint_Q |Df|^p \right)^{1/p} \le C \left( \Barint_{2Q} |Df|^n \right)^{1/n}
\end{equation}
which is assumed to hold for a fixed constant $C>0$ and for all cubes $Q \subset \Omega$ with $\diam(Q) < \dist(Q,\partial\Omega)$ and $\diam(f(2Q)) < \dist(f(2Q),\partial\Omega')$. Here we denote by $2Q$ a cube concentric with $Q$, whose side length equals twice that of $Q$. Note that \eqref{eq:pnK0} implies \eqref{eq:pnK} with constant $C_{(1.2)} = 2C_{(1.1)}$. We refer the reader to \cite[Section 14.4]{im:gft} for additional information on the two versions of a reverse H\"older inequality for quasiconformal mappings.

For $n \ge 2$ and $K \ge 1$, denote by
$$
p_O^{\RH}(n,K)
$$
the supremum of those values $p>n$ so that \eqref{eq:pnK} holds true for every quasiconformal map $f$ between domains in $\R^n$ with $K_O(f) \le K$. We include the superscript in the notation for this higher integrability exponent to indicate the role of the reverse H\"older inequality in the definition. We also denote by
$$
p_O^{\Sob}(n,K)
$$
the supremum of those values $p>n$ so that every quasiconformal map $f$ between domains $\Omega$ and $\Omega'$ in $\R^n$ with $K_O(f) \le K$ necessarily lies in $W^{1,p}_\loc(\Omega:\R^n)$. The value $p_O^{\Sob}(n,K)$ is what is traditionally referred to as the higher integrability exponent for quasiconformal mappings in the literature, where it is generally denoted $p(n,K)$. Note that we always have
\begin{equation}\label{eq:pRHSobnK}
p_O^{\RH}(n,K) \le p_O^{\Sob}(n,K)
\end{equation}
and moreover,
\begin{equation}\label{eq:pRHSobnK2}
p_O^{\Sob}(n,K) \le \frac{nK}{K-1}
\end{equation}
as witnessed by the radial stretch map $f(x) = |x|^{1/K-1}x$. In this paper, the distinction between the (potentially different) values of the higher integrability exponent, as defined in terms of either reverse H\"older inequalities or Sobolev regularity, is important in the statements and proofs of our main results. However, conjecturally, equality holds in both \eqref{eq:pRHSobnK} and \eqref{eq:pRHSobnK2} for all $n \ge 3$ and $K \ge 1$.

In a similar fashion, let $p_I^{\RH}(n,K)$, respectively $p_I^{\Sob}(n,K)$, denote the supremum of values so that \eqref{eq:pnK} holds true, respectively $f$ lies in $W^{1,p}_\loc(\Omega:\R^n)$, for every quasiconformal map $f$ between domains $\Omega$ and $\Omega'$ in $\R^n$ with $K_I(f) \le K$. Here $K_I(f)$ denotes the inner dilatation of $f$. Recall that $K_I(f) \le K_O(f)^{n-1}$ and $K_O(f) \le K_I(f)^{n-1}$ in any dimension. In particular, when $n=2$ the outer and inner dilatations coincide ($K_O(f)=K_I(f)=:K_f$) and the higher integrability exponents $p^{\RH}(2,K)$ and $p^{\Sob}(2,K)$ are defined without reference to either the outer or inner dilatation. Astala \cite{Astala} showed that
\begin{equation}\label{eq:p2K}
p^{\Sob}(2,K) = \frac{2K}{K-1}
\end{equation}
and the stronger statement
\begin{equation}\label{eq:p2K2}
p^{\RH}(2,K) = \frac{2K}{K-1}
\end{equation}
can be found in \cite[Theorem 13.2.3 and Corollary 13.2.4]{aim:2qc}.\footnote{Note, however, that the identity in \eqref{eq:p2K2} was surely already known to Astala, see e.g.\ \cite[Corollary 3.4]{Astala}.}

Returning to the results of \cite{GV}, let $E$ be a subset of $\Omega$ with $\dim_H(E) = \alpha \in (0,n)$. Here and henceforth we denote by $\dim_H(E)$ the Hausdorff dimension of a set $E$. Then
\begin{equation}\label{eq:gv}
0< \frac{(p_O^{\Sob}(n,K^{n-1})-n)\alpha}{p_O^{\Sob}(n,K^{n-1})-\alpha} \le \dim_H f(E) \le \frac{p_O^{\Sob}(n,K)\alpha}{p_O^{\Sob}(n,K) - n + \alpha} < n
\end{equation}
for any $K$-quasiconformal map $f:\Omega \to \R^n$. In particular, quasiconformal maps in $\R^n$ preserve the dimension of sets of Hausdorff dimension $0$ or $n$. Note that the lower bound in \eqref{eq:gv} is obtained by applying the upper bound proved in \cite{GV} to the inverse map $g = f^{-1}$, since $g$ is again quasiconformal with $K_O(g)=K_I(f)\le K_O(f)^{n-1}\leq K^{n-1}$. The two-sided estimates in \eqref{eq:gv} are sometimes written in the `symmetric' form
\begin{equation}\label{eq:gv2} \small
\left( 1 - \frac{n}{p_O^{\Sob}(n,K)} \right) \left( \frac1{\dim_H E} - \frac1n \right) \le \frac1{\dim_H f(E)} - \frac1n \le \left( 1 - \frac{n}{p_O^{\Sob}(n,K^{n-1})} \right)^{-1} \left( \frac1{\dim_H E} - \frac1n \right),
\end{equation}
illustrating the role of the local H\"older exponent $1-n/p$ for $W^{1,p}_\loc(\Omega:\R^n)$ mappings in dimension $n$. In particular, when $n=2$ we have
\begin{equation}\label{eq:astala2}
\frac1K \, \left( \frac1{\dim_H E} - \frac12 \right) \le \frac1{\dim_H f(E)} - \frac12 \le K \, \left( \frac1{\dim_H E} - \frac12 \right),
\end{equation}
for $E \subset \R^2$ and $f$ a quasiconformal map of $\R^2$, as observed by Astala \cite{Astala}.

A panoply of metrically defined notions of dimension have been introduced to elucidate disparate features of sets and metric spaces. These include, for instance, box-counting (Minkowski) dimension and its countably stable regularization, packing dimension, as well as Assouad dimension. We denote by $\ovdimB(E)$ the upper box-counting dimension of a bounded set $E \subset \R^n$ and by $\dim_A(E)$ the Assouad dimension of an arbitrary set $E \subset \R^n$. We refer to section \ref{sec:background} for the full definitions of these and other notions of dimension, but for later purposes in this introduction we remind the reader that
$$
\dim_A(E) = \inf \{ s>0 \, : \, \mbox{$E$ is $s$-homogeneous} \},
$$
where a set $E$ is said to be $s$-homogeneous (with $s$-homogeneity constant $C$) if the number $N(B(x,R)\cap E,r)$ of small sets of diameter at most $r$ needed to cover a large ball $B(x,R) \cap E$ is bounded above by $C(R/r)^s$, uniformly for $0<r\le R$ and $x \in E$. We recall that
$$
\dim_H(E) \le \ovdimB(E) \le \dim_A(E)
$$
for all bounded $E$, and
$$
\dim_H(E) \le \dim_A(E)
$$
for all $E$. The Assouad dimension is unchanged upon taking the closure of a set (Proposition \ref{prop:Assouad-spectrum-facts}(1)(c)); this fact allows us to focus on all sets instead of closed sets and on bounded sets instead of compact sets. Assouad dimension was introduced in connection with the existence (and non-existence) of bi-Lipschitz embeddings into Euclidean spaces. The past two decades have witnessed the increased role of Assouad dimension in quasisymmetric uniformization questions, with particular emphasis on quasisymmetric uniformization of metric $2$-spheres. There is also substantial interest in Assouad dimension for its own sake, as a tool for the study of the metric geometry of Euclidean sets and sets in more general metric spaces. We refer the reader to \cite{Fraser2020} for a comprehensive study of Assouad dimension from the perspective of fractal geometry.

In \cite{tys:assouad}, motivated by some (at that time unresolved) questions regarding dilatation-independent distortion of Hausdorff dimension by  quasiconformal maps, the second author studied analogous questions for Assouad dimension. While the primary focus of \cite{tys:assouad} was on dilatation-independent results, the following dilatation-dependent analog for \eqref{eq:gv} was included. Let $f:\R^n \to \R^n$ be a $K$-quasiconformal map ($n \ge 2$) and let $E \subset \R^n$ satisfy $\dim_A(E) = \alpha \in (0,n)$. Then
\begin{equation}\label{eq:assouad-first}
0<\beta_- \le \dim_A f(E) \le \beta_+ < n
\end{equation}
where the constants $\beta_\pm$ depend only on $n$, $K$, $\alpha$, and (in the case of $\beta_+$) an $s$-homogeneity constant for the set $E$ for some exponent $s$, $\alpha < s < n$. The proof for this result was quite different from the classical proof by Gehring and V\"ais\"al\"a; the lower bound was derived using the fact that Euclidean quasiconformal mappings are power quasisymmetric, and the upper bound relied on the connection between Assouad dimension and porosity and the quasiconformal invariance of porosity. No explicit bounds for $\beta_\pm$ were given, and it was left as an open question whether the stated dependence of $\beta_+$ on the auxiliary homogeneity constant was necessary.

Our first main result (Theorem \ref{th:main1}) addresses both of the above issues. We give precise estimates for the upper and lower bounds in \eqref{eq:assouad-first} and we show that the upper bound can be chosen independent of any auxiliary homogeneity data.

\begin{theorem}\label{th:main1}
Let $f:\Omega \to \Omega'$ be a $K$-quasiconformal map between domains in $\R^n$, $n \ge 2$. Let $E \subset \Omega$ be a compact set satisfying $0<\dim_A(E)<n$. Then
\begin{equation}\label{eq:main1}\small
\left( 1 - \frac{n}{p_O^{\RH}(n,K)} \right) \left( \frac1{\dim_A E} - \frac1n \right) \le \frac1{\dim_A f(E)} - \frac1n \le \left( 1 - \frac{n}{p_O^{\RH}(n,K^{n-1})} \right)^{-1} \left( \frac1{\dim_A E} - \frac1n \right).
\end{equation}\normalsize
If $\Omega = \Omega' = \R^n$ then the conclusion holds for all subsets $E$ (not necessarily compact).
\end{theorem}

We do not know whether the estimates in \eqref{eq:main1} hold, for $n\ge 3$, with the roles of $p_O^{\RH}(n,K)$ and $p_O^{\RH}(n,K^{n-1})$ replaced by $p_O^{\Sob}(n,K)$ and $p_O^{\Sob}(n,K^{n-1})$. Assouad dimension is a global, scale-invariant measurement; in order to control its distortion under quasiconformal mappings we need to impose corresponding global and scale-invariant control on the higher integrability exponent. However, the conclusion in Theorem \ref{th:main1} does suffice to yield a quantitative conclusion regarding dimension distortion. See Corollary \ref{cor:main1} for details. Moreover, as previously observed, we conjecture that the exponents $p_O^{\RH}(n,K)$ and $p_O^{\Sob}(n,K)$ coincide for all relevant $n$ and $K$.

While the basic formulation of the upper and lower estimates in Theorem \ref{th:main1} aligns with that in the Gehring--V\"ais\"al\"a\ result (with Hausdorff dimension replaced by Assouad dimension), the proofs are quite different. In particular, the upper bound
\begin{equation}\label{eq:sob-bound}
\dim_H f(E) \le \frac{p\alpha}{p-n+\alpha}
\end{equation}
as in \eqref{eq:gv} holds for any set $E \subset \Omega \subset \R^n$ with $\dim_H(E) = \alpha$ and any $W^{1,p}_\loc(\Omega:\R^n)$ mapping $f$ (not necessarily quasiconformal, or even a homeomorphism). As previously mentioned, distortion of Hausdorff and box-counting (as well as packing) dimensions by Sobolev mappings was studied by Kaufman \cite{Kaufman}. Our proof of Theorem \ref{th:main1} explicitly uses the fact that $f$ is a quasiconformal homeomorphism, and not just its modulus of uniform continuity or membership in a suitable Sobolev space. Observe that Lipschitz mappings may increase Assouad dimension, see \cite[Example A.6.2]{Luukkainen} or \cite[Theorem 10.2.5]{Fraser2020} for examples.

Our second main theorem concerns the distortion of the Assouad spectrum by Euclidean quasiconformal maps. The Assouad spectrum, introduced by Fraser and Yu \cite{fy:assouad-spectrum}, is a one-parameter family of metrically defined dimensions which interpolates between the upper box-counting dimension and the (quasi-)Assouad dimension. Specifically, the Assouad spectrum of a set $E \subset \R^n$ is a collection of values
$$
\{\dim_A^\theta(E):0<\theta<1\},
$$
where $\dim_A^\theta(E)$ captures the growth rate of the covering number $N(B(x,R)\cap E,r)$ for scales $0<r\le R<1$ related by $R = r^\theta$. The map $\theta \mapsto \dim_A^\theta(E)$ is continuous (even locally Lipschitz) when $0<\theta<1$, and
$$
\dim_A^\theta(E) \to \ovdimB(E) \quad \mbox{as $\theta \to 0$}, \qquad \dim_A^\theta(E) \to \dim_{qA}(E) \quad \mbox{as $\theta \to 1$} \, .
$$
Here $\dim_{qA}(E)$ denotes the {\it quasi-Assouad dimension} of $E$, a variant of Assouad dimension introduced by L\"u and Xi [LX]. We always have $\dim_{qA}(E) \le \dim_A(E)$, and equality holds in many situations (see \cite[Section 3.3]{Fraser2020} for details).

In fact, we use a slightly modified version of the Assouad spectrum where the relationship $R=r^\theta$ between the two scales is relaxed to an inequality $R\ge r^\theta$. This modification leads to the notion of {\it upper Assouad spectrum}, denoted $\{\overline{\dim_A^\theta}(E):0<\theta<1\}$ in the literature. See \cite{fhhty:two} or \cite[Section 3.3.2]{Fraser2020} for more information. For a fixed $\theta \in (0,1)$ the key relationship between the respective two values in the two spectra (see Theorem 3.3.6 in \cite{Fraser2020}) is that
\begin{equation}\label{eq:regularization}
\overline{\dim_A^\theta}(E) = \sup_{0<\theta'<\theta} \dim_A^{\theta'}(E).
\end{equation}
In this paper, we propose the term {\it regularized Assouad spectrum} in lieu of upper Assouad spectrum, and use the notation $\dim_{A,reg}^\theta(E)$ in place of $\overline{\dim_A^\theta}(E)$.

For $t>0$ define
$$
\theta(t) = \frac1{1+t}.
$$

Note that the next two statements still hold formally when one of the denominators is equal to $0$, with the convention that $1/0=\infty$. For example, if $\dim_{A,reg}^{\theta(t/K)}(E)=0$, then Theorem 1.4 implies that $\dim_{A,reg}^{\theta(t)}(f(E))$ needs to be equal to $0$ as well.

\begin{theorem}\label{th:main2}
Let $f:\Omega\to\Omega'$ be a $K$-quasiconformal map as in Theorem \ref{th:main1}, For $t>0$ and a compact set $E$ contained in $\Omega$, we have
\begin{equation}\label{eq:main2}\begin{split}
\small
&\left( 1 - \frac{n}{p_O^{\RH}(n,K)} \right) \left( \frac1{\dim_{A,reg}^{\theta(t/K)}(E)} - \frac1n \right) \\
& \qquad \le \frac1{\dim_{A,reg}^{\theta(t)}(f(E))} - \frac1n \le 
\\
&\left( 1 - \frac{n}{p_O^{\RH}(n,K^{n-1})} \right)^{-1} \left( \frac1{\dim_{A,reg}^{\theta(Kt)}(E)} - \frac1n \right) \, .
\end{split}\end{equation}
\end{theorem}

\noindent Using the known sharp value $p^{\RH}(2,K)=2K/(K-1)$ we obtain the following corollary.

\begin{corollary}\label{cor:main2}
Let $f:\Omega\to\Omega'$ be a $K$-quasiconformal map between domains in $\C$. For $t>0$ and a compact set $E \subset \Omega$, we have
$$
\frac1K \left( \frac1{\dim_{A,reg}^{\theta(t/K)}(E)} - \frac12 \right) \le \frac1{\dim_{A,reg}^{\theta(t)}(f(E))} - \frac12 \le K \left( \frac1{\dim_{A,reg}^{\theta(Kt)}(E)} - \frac12 \right) \, .
$$
\end{corollary}

Note that the conclusion of Theorem \ref{th:main2} remains restricted to compact sets, even in the case when the quasiconformal mapping in question is globally defined. The reason for this distinction between Theorems \ref{th:main1} and \ref{th:main2} is that the Assouad dimension is M\"obius invariant, which allows us to reduce from the general case to the case of compact sets via inversion in a small sphere in the complementary region.

This paper is organized as follows. Section \ref{sec:background} reviews the precise definitions for, and basic properties of, the Assouad dimension and the (regularized) Assouad spectrum. Section \ref{sec:spirals-classification} indicates how to derive the classification theorem for polynomial spirals, Theorem \ref{th:spirals-classification}, from Corollary \ref{cor:main2}. In Section \ref{sec:qc-dist-Assouad} we prove Theorems \ref{th:main1} and \ref{th:main2}. Section \ref{sec:questions} contains open questions and further remarks motivated by this study.

\vspace{0.1in}

\paragraph{\bf Acknowledgements.} The second author acknowledges support from the National Science Foundation under grant DMS-1600650, `Mappings and Measures in Sub-Riemannian and Metric Spaces' and from the Simons Foundation under grant \#852888, `Geometric mapping theory and geometric measure theory in sub-Riemannian and metric spaces'. In addition, this material is based upon work conducted while the second author is serving at the National Science Foundation. Any opinion, findings, and conclusions or recommendations expressed in this material are those of the authors and do not necessarily reflect the views of the National Science Foundation.

The authors gratefully acknowledge the detailed and helpful comments provided by the referee, which have substantially improved the paper.

\section{Background}\label{sec:background}

\subsection{Quasiconformal mappings.}\label{subsec:qc-mappings}

A homeomorphism $f:\Omega \to \Omega'$ between domains in $\R^n$, $n \ge 2$, is said to be {\it $K-$quasiconformal}, $K\geq1$, if $f$ lies in the local Sobolev space $W^{1,n}_\loc(\Omega:\R^n)$ and the inequality
\begin{equation}\label{eq:qc-defn}
|Df|^n \le K \det Df
\end{equation}
holds a.e.\ in $\Omega$. Here $Df$ denotes the (a.e.\ defined) differential matrix and $|\bA|=\max\{|\bA(\bv)|:|\bv|=1\}$ denotes the operator norm of a matrix $\bA$. The smallest value $K \ge 1$ for which \eqref{eq:qc-defn} holds a.e.\ in $\Omega$ is known as the {\it outer dilatation} of $f$ and is denoted $K_O(f)$. Alternatively, set $\ell(\bA):=\min\{|\bA(\bv)|:|\bv|=1\}$ and replace \eqref{eq:qc-defn} with the inequality
\begin{equation}\label{eq:qc-defn-2}
\det Df \le K \ell(Df)^n \,.
\end{equation}
The smallest choice of $K$ for which \eqref{eq:qc-defn-2} holds a.e.\ is the {\it inner dilatation} of $f$ and is denoted $K_I(f)$. These two dilatation measures are related by the mutual inequalities $K_O(f) \le K_I(f)^{n-1}$ and $K_I(f) \le K_O(f)^{n-1}$; thus $K_O(f)=K_I(f):=K_f$ when $n=2$.

For $x \in \R^n$ and $r>0$ we denote by $B(x,r)$ the (closed) Euclidean ball with center $x$ and radius $r$, and by $Q(x,r)$ the axes-parallel closed cube centered at $x$ with side length $2r$. In other words, if $x=(x_1,\ldots,x_n)$, then $Q(x,r) = [x_1-r,x_1+r] \times \cdots \times [x_n-r,x_n+r]$. Alternatively, $Q(x,r)$ is the metric ball in the $\ell^\infty$ norm on $\R^n$, with center $x$ and radius $r$.

We record the following elementary inclusions, valid for any $x \in \R^n$ and $r>0$:
\begin{equation}\label{eq:QB}
Q(x,\tfrac1{\sqrt{n}}r) \subset B(x,r) \subset Q(x,r).
\end{equation}
For $\lambda>0$ and a ball $B = B(x,r)$ (resp.\ cube $Q = Q(x,r)$), we denote by $\lambda B$ (resp.\ $\lambda Q$) the set obtained by dilating with scale factor $\lambda$, i.e., $\lambda B = B(x,\lambda r)$ and $\lambda Q = Q(x,\lambda r)$.

We recall the notion of Whitney decomposition of a domain $\Omega \subsetneq \R^n$. For each $k \ge 1$, and for such a domain $\Omega$, we can write $\Omega$ as an essentially disjoint union of closed cubes $\cW = \{ Q_i \}_{i \in I}$, where each cube $Q_i$ satisfies
\begin{equation}\label{eq:whitney}
\frac1{4k} \dist(Q_i,\partial\Omega) \le \diam(Q_i) \le \frac1k \dist(Q_i,\partial\Omega).
\end{equation}
See, e.g., \cite{ste:singular}.

The following property of quasiconformal mappings is popularly known as the `egg yolk principle'. See, for instance, \cite[Theorem 11.14]{hei:lectures}.

\begin{proposition}\label{prop:egg-yolk}
Fix $n \ge 2$ and $K \ge 1$. Then there exists an increasing homeomorphism $\eta = \eta_{K,n}$ of $[0,\infty)$ so that for any $K$-quasiconformal homeomorphism $f:\Omega \to \Omega'$ between domains in $\R^n$ and any cube $Q \subset \Omega$ with $\diam(Q) \le \dist(Q,\partial\Omega)$, the restriction $f|_Q$ is $\eta$-quasisymmetric. More precisely, if $x,y,z \in Q$ and $x \ne z$, then
$$
\frac{|f(x)-f(y)|}{|f(x)-f(z)|} \le \eta \left( \frac{|x-y|}{|x-z|} \right) \, .
$$
\end{proposition}

The following corollary is standard, but for the benefit of the reader we provide a short proof.

\begin{corollary}\label{cor:egg-yolk}
For each $n \ge 2$, $K \ge 1$, and $c>1$, there exists $k>1$ so that if $f:\Omega \to \Omega'$ is a $K$-quasiconformal mapping between domains in $\R^n$ and $Q \subset \Omega$ is a cube with $\diam(Q) \le \tfrac1k \dist(Q,\partial\Omega)$, then there exists a cube $Q' \subset \Omega'$ so that $\diam(Q') \le \tfrac1c \dist(Q',\partial\Omega')$ and
\begin{equation}\label{eq:egg-yolk}
f(Q) \subset Q' \subset 2Q' \subset f(kQ).
\end{equation}
Similarly, if $B=B(x,R) \subset \Omega$ is a ball with $\diam(B) \le \tfrac1k \dist(B,\partial\Omega)$, then there exists a ball $B'=B(f(x),R') \subset \Omega'$ so that $\diam(B') \le \tfrac1c \dist(B',\partial\Omega')$ and
\begin{equation}\label{eq:egg-yolk2}
	f(B) \subset B' \subset 2B' \subset f(kB).
\end{equation}
\end{corollary}

\begin{proof}
We will prove the corresponding statement for balls instead of cubes; the analogous result for cubes follows immediately upon appeal to \eqref{eq:QB} (increasing the value of $k$ if necessary).

\vspace{0.1in}

With $\eta$ as in Proposition \ref{prop:egg-yolk}, choose $k \ge 1$ so that $\eta(1/k) < 1/(3c)$. Let $B=B(x,R)$ and assume that $\diam(B) \le \tfrac1k \dist(B,\partial\Omega)$. Then $2kB \subset \Omega$. Set $B'=B(f(x),R')$ where
$$
R' := \max\{|f(x)-f(y)|\,:\,|x-y| = R\} \, .
$$
We claim that $3cB' \subset \Omega'$ (which in turn implies that $\diam(B') \le \tfrac1c \dist(B',\partial\Omega')$) and
$$
f(B) \subset B' \subset 2B' \subset f(kB).
$$
The inclusion $f(B) \subset B'$ is immediate from the definition of $R$. Fix $y,z \in \Omega$ so that $|x-y|=R$ and $|f(x)-f(y)|=R'$, while $|x-z|=kR$ and
\begin{equation}\label{eq:z}
|f(x)-f(z)| = \min\{|f(x)-f(w)|\,:\,|x-w| = kR\} \, .
\end{equation}
Appealing to Proposition \ref{prop:egg-yolk}, we obtain
$$
\frac{|f(x)-f(y)|}{|f(x)-f(z)|} \le \eta \left( \frac{|x-y|}{|x-z|} \right) = \eta(\frac1k) < \frac1{3c}.
$$
Since $2kB \subset \Omega$ we conclude from \eqref{eq:z} that $3cR' = 3c|f(x)-f(y)| \le |f(x)-f(z)| < \dist(f(x),\partial\Omega')$. Hence $3cB' \subset \Omega'$. Moreover,
$$
2B' = B(f(x),2|f(x)-f(y)|) \subset B(f(x),\tfrac{2}{3c}|f(x)-f(z)|) \subset B(f(x),|f(x)-f(z)|)\subset f(kB). \qedhere
$$
\end{proof}

In fact, we will use a slight extension of Corollary \ref{cor:egg-yolk}. While the proof of the following statement is elementary, we are not aware of any explicit reference. For the convenience of the reader, and in order to make the proofs in later sections clearer, we provide the statement and a short proof.

\begin{corollary}\label{cor:egg-yolk-2}
For each $n \ge 2$, $K \ge 1$, and $c>1$, there exists $k>1$ so that if $f:\Omega \to \Omega'$ is as in Corollary \ref{cor:egg-yolk}, $E \subset \Omega$ is a set with $\diam(f(E))<\tfrac12$, and $B=B(x,R) \subset \Omega$ is a ball with $x\in E$, $R<1$, and $\diam(B) \le \tfrac1k \dist(B,\partial\Omega)$, then there exists a ball $B'=B(f(x),R') \subset \Omega'$ so that 
\begin{itemize}
\item[(i)] $R'<1$, 
\item[(ii)] $\diam(B') \le \tfrac1c \dist(B',\partial\Omega')$ and
\item[(iii)] $f(B\cap E) \subset B' \cap f(E) \subset 2B'\cap f(E) \subset f(kB \cap E)$.
\end{itemize}
\end{corollary}

\begin{proof}
Using Corollary \ref{cor:egg-yolk}, select a ball $B''=B(f(x),R'')$ with $\diam B'' \le \tfrac1c \dist(B'',\partial\Omega')$ and $f(B) \subset B'' \subset 2B'' \subset f(kB)$. Fix $\delta>0$ so that $\diam(f(E))<\tfrac12-\delta$ and set
\begin{equation}\label{eq:R-R'}
R' = \min\{R'',\tfrac12-\delta\}
\end{equation}
and $B' = B(f(x),R')$. It is clear that (i) holds true. The validity of (ii) follows since $B' \subset B''$. Finally, since $f(E) \subset B(f(x),\tfrac12-\delta)$ we have $B'\cap f(E) = B'' \cap f(E)$ and $2B'\cap f(E) = 2B'' \cap f(E)$, and so
$$
f(B \cap E) \subset B'' \cap f(E) = B' \cap f(E) \subset 2B'' \cap f(E) = 2B' \cap f(E) \subset f(kB \cap E).  \qedhere
$$
\end{proof}

\subsection{Assouad dimension and the Assouad spectrum.}\label{subsec:assouad}

Let $F$ be a bounded subset of $\R^n$. For $r>0$, denote by $N(F,r)$ the smallest number of sets of diameter at most $r$ needed to cover $F$. The {\it (upper) box-counting dimension} of $F$ is
$$
\ovdimB(F) = \limsup_{r\to 0} \frac{\log N(F,r)}{\log(1/r)}.
$$
We drop the adjective `upper' throughout this paper as we will make no reference to the lower box-counting dimension. An equivalent formulation is
$$
\ovdimB(F) = \inf \{\alpha>0 \,:\, \exists\,C>0\mbox{ s.t. } N(F,r) \le C r^{-\alpha} \mbox{ for all $0<r\le\diam(F)$} \}.
$$
Since any bounded set can be covered by finitely many balls of radius $1$, another equivalent formulation is
$$
\ovdimB(F) = \inf \left\{\alpha>0 \,:\, {\exists\,C>0\mbox{ s.t. } N(B(x,1) \cap F,r) \le C (1/r)^{\alpha} \atop \mbox{ for all $0<r\le 1$ and all $x \in F$}} \right\},
$$
and here we observe that the expression on the right hand side also makes sense for unbounded $F$. The latter formulation also shows the connection between box-counting and Assouad dimension. For an arbitrary (not necessarily bounded) set $F \subset \R^n$, the {\it Assouad dimension} of $F$ is
$$
\dim_A(F) = \inf \left\{\alpha>0 \,:\, {\exists\,C>0\mbox{ s.t. } N(B(x,R) \cap F,r) \le C (R/r)^{\alpha} \atop \mbox{ for all $0<r\le R$ and all $x \in F$}} \right\}.
$$
Assouad dimension first appeared (under a different name) in a 1983 paper of Assouad on metric embedding problems \cite{Assouad83}. Luukkainen \cite{Luukkainen} gave a detailed presentation of the state-of-the-art in the theory of Assouad dimension as of the late 1990s. We also recommend the recent book by Fraser \cite{Fraser2020}. It is perhaps important to specify that in the aforementioned book the author defines $N(F,r)$ to be the smallest number of open sets of diameter at most $r$ needed to cover $F$, and uses this covering number in the definition of all dimensions and spectra. It is an elementary exercise to prove that this yields the same exact notions of dimension and spectrum.

It is clear from the definitions that $\ovdimB(F) \le \dim_A(F)$ for bounded sets $F$. Both box-counting dimension and Assouad dimension are monotonic and finitely stable, but neither quantity is countably stable. Indeed, both notions of dimension are invariant under passing to the closure, thus the dimension of any dense subset of $\R^n$ is equal to $n$. Such behavior is in sharp contrast to that exhibited by Hausdorff dimension, which is zero for any countable set. An illustrative example is $F = \{0\} \cup \{m^{-1}\,:\,m \in \N\} \subset \R$, for which $\dim_H(F) = 0$, $\ovdimB(F) = \tfrac12$, and $\dim_A(F) = 1$.

Fraser and Yu \cite{fy:assouad-spectrum} introduced the Assouad spectrum as an interpolation between upper box-counting dimension and Assouad dimension. As discussed in the introduction, we employ a slight modification of the definition, as discussed in \cite{fhhty:two} and \cite[Section 3.3.2]{Fraser2020}. For $0<\theta<1$ and a set $F \subset \R^n$, define
\begin{equation}\label{def:Assouad-spectrum}
\dim_{A,reg}^\theta(F) = \inf \left\{\alpha>0 \,:\, {\exists\,C>0\mbox{ s.t. } N(B(x,R) \cap F,r) \le C (R/r)^{\alpha} \atop \mbox{ for all $0<r\le R^{1/\theta}< R< 1$ and all $x \in F$}} \right\}.
\end{equation}
Thus $\dim_{A,reg}^\theta(F)$ is defined by the same process as $\dim_A(F)$, but with the restriction that the two scales $r$ and $R$ involved in the definition of the latter are related by the inequality $R \ge r^\theta$. The set of values $\{ \dim_{A,reg}^\theta(F) \,:\, 0<\theta<1\}$ is called the {\it regularized Assouad spectrum} of $F$. We collect various properties of the regularized Assouad spectrum in the following proposition. Proofs of these results for the (original) Assouad spectrum can be found in \cite[Sections 3.3 and 3.4]{Fraser2020}, and the corresponding results for the regularized Assouad spectrum follow easily from the regularization identity \eqref{eq:regularization}.\\

\begin{proposition}\label{prop:Assouad-spectrum-facts}
The regularized Assouad spectrum enjoys the following features.
\begin{itemize}
\item[(1)] For fixed $\theta$, the set function $F \mapsto \dim_{A,reg}^\theta(F)$ is
\begin{itemize}
\item[(a)] monotonic, i.e., $E \subset F$ implies $\dim_{A,reg}^\theta(E) \le \dim_{A,reg}^\theta(F)$,
\item[(b)] finitely stable, i.e., $\dim_{A,reg}^\theta(E\cup F) = \max \{ \dim_{A,reg}^\theta(E),\dim_{A,reg}^\theta(F) \}$,
\item[(c)] invariant under taking closures, and
\item[(d)] invariant under bi-Lipschitz transformation.
\end{itemize}
\item[(2)] For a fixed set $F$, the function $\theta \mapsto \dim_{A,reg}^\theta(F)$ is nondecreasing, continuous on $(0,1)$, and Lipschitz on compact subsets of $(0,1)$.
\item[(3)] For fixed $F$, $\lim_{\theta\to 0^+}\dim_{A,reg}^\theta(F)$ exists and equals $\ovdimB(F)$. Moreover, $\lim_{\theta\to 1^-}\dim_{A,reg}^\theta(F)$ coincides with the so-called {\it quasi-Assouad dimension} of $F$, denoted $\dim_{qA}(F)$. For any $\theta$, one has $\dim_{A,reg}^\theta(F) \le \dim_{qA}(F) \le \dim_A(F)$. If $\dim_{A,reg}^\theta(F) = \dim_{qA}(F)$ for some $0<\theta<1$, then $\dim_{A,reg}^{\theta'}(F) = \dim_{qA}(F)$ for all $\theta\le\theta'<1$.
\item[(4)] Set
\begin{equation}\label{eq:rho-definition}
\rho = \rho(F) := \inf\{\theta \in (0,1):\dim_{A,reg}^\theta(F) = \dim_{qA}(F)\},
\end{equation}
or $\rho = 1$ if no such $\theta$ exists. Then
$$
\dim_{A,reg}^\theta(F) \ge (\tfrac{1-\rho}{1-\theta}) \dim_{qA}(F)
$$
for all $0<\theta<\rho$.
\item[(5)] For $F$ bounded and $0<\theta<1$,
\begin{equation}\label{eq:fraser-bound}
\dim_{A,reg}^\theta(F) \le \frac{\ovdimB(F)}{1-\theta} \, .
\end{equation}
Hence the phase transition $\rho(F)$ defined in \eqref{eq:rho-definition} satisfies
\begin{equation}\label{eq:rho-inequality}
\rho(F) \ge 1-\frac{\ovdimB(F)}{\dim_{qA}(F)}\,.
\end{equation}
\end{itemize}
\end{proposition}

Note that the Assouad spectrum function $\theta \mapsto \dim_A^\theta(F)$ is {\bf not} always monotonically increasing. For an example, see \cite[Section 3.4.4]{Fraser2020}.

In the proof of our main theorems in section \ref{sec:qc-dist-Assouad}, it will be convenient to take advantage of several alternative descriptions for the regularized Assouad spectrum. We collect several such descriptions in the following proposition. In part (ii) of the proposition we make use of the standard dyadic decomposition. Specifically, for a set $F \subset \R^n$ and $x \in F, R>0$ we consider the axes-parallel cube $Q:=Q(x,R) \subset \R^n$, we subdivide $Q$ into $2^n$ essentially disjoint subcubes, each with side length equal to half of the side length of $Q$, and then we subdivide each of those cubes in the same fashion, and so on. Let $\cW(Q)$ denote the collection of all such cubes obtained at any level of the construction, and let $\cW_m(Q)$ denote the collection of all cubes obtained after $m$ steps. We will denote by $N_d(B(x,R) \cap F,m)$ the number of dyadic cubes in $\cW_m(Q)$ needed to cover $B(x,R) \cap F$.

\begin{proposition}\label{prop:Assouad-spectrum-technical-proposition}
Let $F \subset \R^n$ be a bounded subset and let $0<\theta<1$.
\begin{itemize}
\item[(i)] Fix a value $C_1 \ge 1$. Then the regularized Assouad spectrum value $\dim_{A,reg}^\theta(F)$ is equal to the infimum of all $\alpha>0$ for which there exists $C>0$ so that
    $$
    N(B(x,R) \cap F,r) \le C(\tfrac{R}{r})^\alpha
    $$
    for all $x \in F$ and all $0<r/C_1\le R^{1/\theta} < R < 1$.

\item[(ii)] The regularized Assouad spectrum value $\dim_{A,reg}^\theta(F)$ is equal to the infimum of all $\alpha>0$ for which there exists $C>0$ so that
    $$
    N_d(B(x,R) \cap F,m) \le C 2^{m\alpha}
    $$
    for all $x \in F$ and all $0<2^{-m}R\le R^{1/\theta} < R < 1$.
\end{itemize}
\end{proposition}

\begin{proof}
	(i) Denote by $$A_\theta :=\left\{\alpha>0 \,:\, {\exists\,C>0\mbox{ s.t. } N(B(x,R) \cap F,r) \le C (R/r)^{\alpha} \atop \mbox{ for all $0<r\le R^{1/\theta}<R<1$ and all $x \in F$}} \right\}$$
	
	and by
	
	$$B_\theta :=\left\{\alpha>0 \,:\, {\exists\,C>0\mbox{ s.t. } N(B(x,R) \cap F,r) \le C (R/r)^{\alpha} \atop \mbox{ for all $0<r/C_1\le R^{1/\theta}<R<1$ and all $x \in F$}} \right\}.$$
	We know by the definition of the regularized Assouad spectrum that $\dim_{A,reg}^\theta(F)=\inf A_\theta$, so it is enough to show that $A_\theta=B_\theta$.
	
	Let $\alpha\in A_\theta$, $x \in F$ and $r, R>0$ with $0<r/C_1\le R^{1/\theta}<R<1$.
	Then $$N(B(x,R)\cap F, r) \leq N(B(x,R)\cap F, r/C_1) \leq  C \left(\frac{R}{r/C_1}\right)^{\alpha}=  C C_1^\alpha (R/r)^{\alpha}$$
	which implies that $\alpha\in B_\theta$. Hence, $A_\theta \subset B_\theta$.
	
	Since $C_1\geq 1$, for any $0<r\leq R^{1/\theta} \le R<1$ we have that $r/C_1 \le r$, which makes the inclusion $B_\theta \subset A_\theta$ trivial. Hence $A_\theta = B_\theta$.
	
	(ii) Denote by $$D_\theta :=\left\{\alpha>0 \,:\, {\exists\,C>0\mbox{ s.t. } N_d(B(x,R) \cap F,m) \le C 2^{m\alpha} \atop \mbox{ for all $m\in \N, R>0$ with $0<2^{-m}R\le R^{1/\theta} < R < 1$ and all $x \in F$}} \right\}.$$
	
	Similarly, it is enough to show that $A_\theta = D_\theta$. For this we will need the inequalities
	\begin{equation}\label{covineq}
		N(B(x,R)\cap F,2^{-m+1}R \sqrt{n}) \leq N_d(B(x,R)\cap F,m) \leq 3^n N(B(x,R)\cap F,2^{-m+1}R).
	\end{equation}
	
	Since cubes of side length $2^{-m+1}R$ are sets of diameter at most $2^{-m+1}R \sqrt{n}$, the left inequality is trivial. For the right inequality it is enough to notice that every set of diameter at most $2^{-m+1}R$ cannot possibly intersect more than $3^n$ axes-paralleled cubes of side length $2^{-m+1}R$.
	
	Let $\alpha\in A_\theta$, $x \in F$ and $m \in \N, R>0$ with $0<2^{-m+1}R\le R^{1/\theta} < R < 1$. Then $$N(B(x,R)\cap F, 2^{-m+1} R) \leq C \left( \frac{R}{2^{-m+1} R}\right)^\alpha$$ which by \eqref{covineq} implies that
	$$N_d(B(x,R)\cap F,m) \leq C 2^{-\alpha} 3^{n} 2^{m\alpha}$$ and, thus, $\alpha \in D_\theta$. Hence $A_\theta \subset D_\theta$.
	
	Let $\alpha \in D_\theta$, $x \in F$ and $r, R>0$ with $0<r\leq R^{1/\theta}<R<1$. Let $m \in \N$ be the smallest number for which
	$2^{-m+1} R \sqrt{n}  \leq r < 2^{-m+2}  R \sqrt{n}.$ Then
	$$ N(B(x,R) \cap F, r) \leq N(B(x,R) \cap F, 2^{-m+1}  R \sqrt{n})$$
	which by \eqref{covineq} implies that
	$$ N(B(x,R) \cap F, r) \leq N_d(B(x,R)\cap F, m) \leq C 2^{m \alpha}=C \left( \frac{R}{2^{-m+2} R \sqrt{n}} \right)^\alpha (2^2 \sqrt{n})^\alpha.$$
	Hence, $N(B(x,R) \cap F, r) \leq C (4 \sqrt{n})^\alpha (R/r)^\alpha$ which means that $\alpha \in A_\theta$. As a result, $D_\theta \subset A_\theta$ and the proof is complete.	
\end{proof}

\section{Quasiconformal classification of polynomial spirals}\label{sec:spirals-classification}

Recall from the introduction that $S_a$ denotes the polynomial spiral
$$
\{ x^{-a} e^{\bi x} \in \C : x>0 \}.
$$
Theorem \ref{th:spirals-classification} asserts that $S_a$ can be mapped to $S_b$ by a quasiconformal map $f$ of $\C$ if and only if $K_f \ge \tfrac{a}{b}$. We will prove this result as an application of Corollary \ref{cor:main2}. Let us first observe why other notions of dimension are insufficient for this purpose. Clearly, since $\dim_H(S_a) = 1$ for every $a>0$, Astala's result \eqref{eq:astala2} cannot be used to quasiconformally distinguish any pair of polynomial spirals.

As noted in the introduction, the estimate \eqref{eq:sob-bound} holds true for any $W^{1,p}_\loc(\Omega:\R^n)$ mapping $f$ (not necessarily quasiconformal) from a domain $\Omega \subset \R^n$ with $p>n$, and for any set $E \subset \Omega$ with $\dim_H E = \alpha \in (0,n)$.
Kaufman \cite{Kaufman} proved the analogous statement for the box-counting dimension $\dim_B$. Using again the sharp exponent of Sobolev integrability for planar quasiconformal maps one concludes that if $f:\C\to\C$ is $K$-quasiconformal and $E\subset \C$ is bounded, then
$$
\frac1K \left( \frac1{\dim_B(E)} - \frac12 \right) \le \frac1{\dim_B(f(E))} - \frac12 \le K \left( \frac1{\dim_B(E)} - \frac12 \right) \, .
$$
One may try to use this estimate to answer the question about quasiconformal equivalence of polynomial spirals. Fraser \cite{fr:spirals} computed the box-counting dimensions of such spirals: for $a>0$,
$\dim_B(S_a) = \max \{ 2/(1+a), 1 \}$. It follows that $\tfrac1{\dim_B(S_a)}-\tfrac12 = \min\{\tfrac{a}{2},\tfrac12\}$. Thus if $f:\C\to\C$ is $K$-quasiconformal with $f(S_a) = S_b$, $a>b>0$, then
$$
\frac1{\dim_B(S_a)} - \frac12 \le K \left( \frac1{\dim_B(S_b)} - \frac12 \right)
$$
and so
$$
K \ge \frac{\min\{a,1\}}{\min\{b,1\}}.
$$
This proves Theorem \ref{th:spirals-classification}, but only in the case $0<b<a\le 1$. If $a>1$ then $\dim_B(S_a) = 1$ and the preceding lower bound for $K$ does not match the upper bound given by the radial stretch map. If $b>1$ then $\dim_B(S_a) = \dim_B(S_b) = 1$ and we obtain no nontrivial information about the dilatation.

To resolve the remaining case, we consider the Assouad spectrum. Fraser \cite{fr:spirals} also computed these quantities for the polynomial spirals. For $a>0$ and $0<\theta<1$,
$$
\dim_A^\theta(S_a) = \begin{cases} \min\left\{ \frac{2}{(1+a)(1-\theta)},2\right\}, & \mbox{if $0<a\le 1$,} \\
\min\left\{ 1+\frac{\theta}{a(1-\theta)},2\right\}, & \mbox{if $a\ge 1$.} \end{cases}
$$
Note that since these expressions are monotonically increasing as functions of $\theta$, they also agree with the respective values $\dim_{A,reg}^\theta(S_a)$ of the regularized Assouad spectrum. The critical parameter $\rho(S_a)$, as in \eqref{eq:rho-definition}, is $a/(1+a)$. Note that equality in \eqref{eq:rho-inequality} holds only when $0<a\le 1$. We have $\dim_{A}^\theta(S_a)<2$ if $\theta<\rho(S_a)$ and $\dim_{A}^\theta(S_a)=2$ if $\theta\ge\rho(S_a)$. It follows that the quasi-Assouad dimension (and hence also the Assouad dimension) of $S_a$ equals $2$ for all $a>0$, and so Assouad dimension cannot be used to distinguish polynomial spirals up to quasiconformal equivalence.

\begin{proof}[Proof of Theorem \ref{th:spirals-classification}]
As discussed above, it suffices to show that if $K<\tfrac{a}{b}$, $a>b>0$, then there does not exist a $K$-quasiconformal map $f$ of $\C$ with $f(S_a) = f_b$.

Suppose that such a map exists. Set $t=1/b$. Then $\theta(t) = 1/(1+t) = b/(1+b)$ so $\dim_A^{\theta(t)}(S_b) = 2$. On the other hand, $\theta(t/K) = K/(K+t) < a/(1+a)$ so $\dim_A^{\theta(t/K)}(S_a) < 2$. This leads to a contradiction with the conclusion of Corollary \ref{cor:main2}. Note that
$\tfrac{K}{K+1/b} < \tfrac{a}{1+a}$ if and only if $K < \tfrac{a}{b}$.
\end{proof}

\begin{remark}
In the preceding proof, distinguishing different polynomial spirals $S_a$ up to $K$-quasiconformal equivalence relies on an understanding of the behavior of the Assouad spectrum values $\dim_{A}^\theta(S_a)$ as a function of $\theta$, and more precisely, determining the threshold parameter $\rho(S_a)$ where the value of $\dim_{A}^\theta(S_a)$ reaches the dimension of the ambient space $\R^2$. The precise form of the upper bound for the regularized Assouad spectrum values of $f(E)$ in terms of the corresponding values for $E$ plays no role. It would be interesting to identify a situation in which the precise bounds in \eqref{eq:main2} feature in a quasiconformal classification problem.
\end{remark}

\section{Quasiconformal distortion of Assouad dimension and the Assouad spectrum}\label{sec:qc-dist-Assouad}

In this section, we prove Theorems \ref{th:main1} and \ref{th:main2}. In subsection \ref{subsec:Assouad-spectrum-proof} we prove Theorem \ref{th:main2} on the quasiconformal distortion of the Assouad spectrum. The proof of Theorem \ref{th:main1}, on the distortion of Assouad dimension, proceeds along similar lines. We present this proof in subsection \ref{subsec:Assouad-dimension-proof} in an abbreviated form, focusing on those aspects of the argument in subsection \ref{subsec:Assouad-spectrum-proof} which must be modified.

\subsection{Proof of Theorem \ref{th:main2}}\label{subsec:Assouad-spectrum-proof}
Recall that our goal here is to prove the dimension distortion estimates \eqref{eq:main2} for any $K$-quasiconformal map $f:\Omega \to \Omega'$ between domains in $\R^n$ and for compact sets $E \subset \Omega$. We begin by performing some preliminary reductions.

If $\Omega = \R^n$ then also $\Omega' = \R^n$. Using the bi-Lipschitz invariance of the Assouad spectrum (Proposition \ref{prop:Assouad-spectrum-facts}(1)(d)), and pre- and post-composing with suitable homotheties, we may assume without loss of generality that $E \subset Q_0$ and $f(E) \subset Q_0$, where
$$
Q_0 = Q\left(0,\tfrac{1}{5\sqrt{n}}\right) = [-\frac{1}{5\sqrt{n}},\frac{1}{5\sqrt{n}}]^n.
$$ 
This choice ensures that both $\diam(E)$ and $\diam(f(E))$ are strictly less than $\tfrac12$, which allows us to use Corollary \ref{cor:egg-yolk-2} in the subsequent discussion.

If $\Omega \subsetneq \R^n$ then also $\Omega' \subsetneq \R^n$. In this case we consider a suitable Whitney decomposition of $\Omega$. Let $\eta$ be a local quasisymmetric distortion function as in Proposition \ref{prop:egg-yolk} and let $c = \max\{\eta(1),3\}$. Choose $k \ge 1$ as in Corollary \ref{cor:egg-yolk} and select a Whitney decomposition $\cW = \{Q_i\}_{i\in I}$ for $\Omega$ satisfying \eqref{eq:whitney} for each $Q \in \cW$. Since $E$ is compact, it has nonempty intersection with only finitely many cubes in the Whitney decomposition of $\Omega$. In view of the finite stability of the Assouad spectrum (Proposition \ref{prop:Assouad-spectrum-facts}(1)(b)) we may without loss of generality assume that $E$ is contained in one such cube $Q \in \cW$. A further appeal to Corollary \ref{cor:egg-yolk} yields a cube $Q' \subset \Omega'$ so that $f(Q) \subset Q'$ and $\diam f(Q) \le \diam Q' \le \tfrac1c \dist(Q',\partial\Omega') \le \tfrac1c \dist(f(Q),\partial\Omega')$. Using again the bi-Lipschitz invariance of the Assouad spectrum, we may assume without loss of generality that $Q = Q' = Q_0$.

In conclusion, and regardless of which of the above two cases holds, we may assume with no loss of generality that both $E$ and $f(E)$ are contained in $Q_0$, that $f(Q_0) \subset Q_0$, and that
\begin{equation}\label{eq:Q0}
\diam Q_0 \le \min\{ \tfrac1k \dist(Q_0,\partial\Omega), \tfrac1c \dist(Q_0,\partial\Omega')\},
\end{equation}
where we interpret the right hand side as $+\infty$ if $\Omega = \Omega' = \R^n$. In particular, $Q_0 \subset \Omega$ and $Q_0 \subset \Omega'$. 

\vspace{0.1in}

We now begin the proof in earnest. It suffices to prove one of the two inequalities in \eqref{eq:main2}, as the other inequality follows by considering the inverse map $f^{-1}$. We prove the left hand inequality, which we rewrite in the form
$$
\dim_{A,reg}^{\theta(t)}(f(E)) \le \beta_0 := \frac{p_O^{\RH}(n,K)\alpha_0}{p_O^{\RH}(n,K)-n+\alpha_0}, \qquad \alpha_0 = \dim_{A,reg}^{\theta(t/K)}(E).
$$
This follows if we prove
$$
\dim_{A,reg}^{\theta(t)}(f(E)) \le \beta := \frac{p\alpha}{p-n+\alpha}
$$
for all $p\in (n,p_O^{\RH}(n,K))$ and $\alpha \in (\alpha_0, n]$.

Fix $p$, $\alpha$ and $\beta$ as above. Let $y \in f(E)$ and $0<R'< 1$. Note that these assumptions imply that $B(y,R') \subset \Omega'$. Indeed, if $z \in B(y,R') \setminus \Omega'$ then
$$
1>R' \ge |y-z| \ge \dist(y,\partial\Omega') \ge \dist(Q_0,\partial\Omega') \ge c \diam Q_0 \ge \tfrac65.
$$
In fact, we can replace any ball centered inside $E$ or $f(E)$ by one with the same center and radius at most $\diam E$ or $\diam f(E)$ that has the same intersection with $E$ or $f(E)$ respectively. By further reducing the diameter of $Q_0$ if necessary, we can assume that all such balls lie further enough from the boundaries $\partial\Omega$ and $\partial\Omega'$ that Corollary 2.3 can be applied to them.\\
We will find a constant $C_1'>0$ and we will cover $f(E) \cap B(y,R')$ by sets of diameter at most $r_m'$ for all $m$ so that $r_m' \le C_1' (R')^{1/\theta(t)}$, where $(r_m')$ is a decreasing sequence with $r_m' \searrow 0$ and $\tfrac{r_{m+1}'}{r_m'} \ge c_0$ for some $c_0 \in (0,1)$ independent of $m$. This covering leads to an estimate for $N(B(y,R') \cap f(E),r_m')$ from above by $C(R'/r_m')^\beta$ for some fixed constant $C$; in view of Proposition \ref{prop:Assouad-spectrum-technical-proposition}(i) we conclude that $\dim_{A,reg}^{\theta(t)}(f(E)) \le \beta$, which finishes the proof.

For the given choice of $y$ and $R'<1$, let $B = B(y,R')$. We apply Corollary \ref{cor:egg-yolk-2} to the inverse map $g = f^{-1}$, which yields a ball $B(x,R)\subset \Omega$ with $x = g(y)$, $R<1$, and $g(B(y,R') \cap f(E)) \subset B(x,R) \cap E$. If we cover $B(x,R) \cap E$ with sets from a covering $\cU$, then the set $B(y,R') \cap f(E) \subset f(B(x,R)) \cap f(E)$ will be covered by the images of the elements in $\cU$.

We consider cubes obtained via dyadic decomposition of $Q(x,R)$. The side lengths of such cubes assume values $2^{-m}R$ for values $m \ge -1$. Let $m_0 = m_0(R)$ be the unique positive integer so that
$$
2^{-m_0}R \le R^{1/\theta(t/K)} = R^{1+t/K} < 2^{-m_0+1}R,
$$
and let
\begin{eqnarray*}
	r_m &:= 2^{-m}R, \\
	r_m' & := 2^{-m\alpha/\beta}R'
\end{eqnarray*}
for $m \ge m_0$. Since $\alpha>\alpha_0$, for any such choice of $m$ we need at most
	\begin{equation}\label{1}
		C_0\left(\frac{R}{r_m}\right)^\alpha = C_0 2^{m\alpha}
	\end{equation}
dyadic cubes of side length $2^{-m}R$ to cover $E \cap B(x,R)$; see Proposition \ref{prop:Assouad-spectrum-technical-proposition}(ii).
	
Recalling the choice of $R$ in \eqref{eq:R-R'} and using standard local H\"older continuity estimates for $f$ and $f^{-1}$ (see, for instance, \cite[Theorem 7.7.1]{im:gft}), we conclude that
$$
R \ge (R')^K/C_1
$$
for some $C_1>0$. It follows that there exists $C_2$ depending only on $C_1$, $t$, and $K$, so that if we denote by $m_0' = m_0'(R')$ the unique integer so that
\begin{equation}\label{C1ineq}
r'_{m'_0}=2^{-m_0'\alpha/\beta}R' \le (1/C_2) (R')^{1/\theta(t)} = (1/C_2) (R')^{1+t} < 2^{(-m_0'+1)\alpha/\beta}R'=r'_{m'_0-1},
\end{equation}
then
$$
m_0' \ge m_0.
$$
To see this, note that since $\alpha \le \beta$ we have
$$
2^{(-m_0+1)\alpha/\beta}R' \ge 2^{-m_0+1}R' > R^{t/K}R' \ge C_1^{-t/K} (R')^{1+t}
$$
so \eqref{C1ineq} holds with $C_2 = C_1^{t/K}$, which means that we have $r'_m \leq (1/C_2) (R')^{1/\theta(t)}$ for all $m\geq m'_0(R')$.
	
Fix an integer $m \ge m_0'$. Following the terminology in \cite{Kaufman}, we call a dyadic cube $Q$ \textbf{minor} if $\diam f(Q) \leq r'_m$ and \textbf{major} otherwise. The distinction between major and minor cubes depends on the choice of the level $m$, however, it applies to the dyadic cubes of all levels. If the cube $Q(x,R)$ is minor, then all dyadic sub-cubes of level $m$ used to cover $E \cap B(x,R)$ will in fact be minor, in which case \eqref{1} bounds the number of sets of diameter at most $r'_m$ needed to cover $f(E) \cap B(y,R')$ by $C_0 2^{m\alpha}$. If $Q(x,R)$ is major then, since $f$ is uniformly continuous, we can subdivide any dyadic cube $Q$ into dyadic minor subcubes of varying sizes, where all dyadic minor subcubes in question have the property that their dyadic parent is major. Let us call such a cube a \textbf{critical (minor) cube}. We will estimate the number of critical cubes by counting their major parents.

\begin{lemma}\label{Le2}
The total number of major cubes of side length at most $2^{-m}R$ contained in $Q(x,R)$ is bounded above by $C 2^{m \alpha}$, where the constant $C$ depends only on $K$ and $n$.
\end{lemma}
	
\begin{proof}
For fixed $j \ge 0$, let $M(j)$ be the number of major cubes in $Q(x,R)$ of side length $2^{-j}R$. Denote by $Q_i^j$ a typical such cube. The Morrey--Sobolev inequality on $Q^i_j$ takes the form
$$
\diam f(Q_i^j) \leq C_3 (\diam (Q_i^j))^{1-n/p} \left( \int_{Q_i^j} |Df|^p \right)^{1/p}.
$$
Since $Q_i^j$ is major, this implies that
$$
2^{-\frac{m\alpha p}{\beta}}(R')^p \leq C_4 \, 2^{-j (p-n)} R^{p-n} \int_{Q_i^j}|Df|^p
$$
from which it follows that
\begin{equation*}\begin{split}
2^{-\frac{m\alpha p}{\beta}}(R')^p M(j)
&\leq C_4 \, 2^{-j(p-n)} R^{p-n} \int_{\bigcup\limits_{i} Q_i^j}|Df|^p \\
&\leq C_4 \, 2^{-j(p-n)} R^{p-n} \int_{Q(x,R)}|Df|^p \,.
\end{split}\end{equation*}
Summing over all $j \geq m$ yields
$$
2^{-\frac{m\alpha p}{\beta}}(R')^p \sum_{j=m}^{\infty} M(j) \leq C_5 \, 2^{-m(p-n)} R^{p-n} \int_{Q(x,R)} |Df|^p.
$$
Hence
\begin{equation}\label{3}
\sum_{j=m}^{\infty} M(j) \leq C_5 \, 2^{\frac{m\alpha p}{\beta}-m(p-n)} \, (R')^{-p}R^{p-n} \int_{Q(x,R)}|Df|^p \,.
\end{equation}
Since the doubled cube $Q(x,2R)$ can also be assumed to be contained in $Q_0$ by performing similar reductions on the set $E$ as in the beginning of the proof, \eqref{eq:Q0} implies that $\diam f(Q(x,2R)) \le \dist(f(Q(x,2R)),\partial\Omega')$; since $p<p^{\RH}_O(n,K)$ we conclude that the reverse H\"older inequality \eqref{eq:pnK} is satisfied, i.e.
$$
\frac{1}{|Q(x,R)|^{1/p}} \left( \int_{Q(x,R)} |Df|^p \right)^{1/p} \leq \frac{C}{|Q(x,2R)|^{1/n}}\left( \int_{Q(x,2R)} |Df|^n \right)^{1/n} \, .
$$
Using \eqref{eq:egg-yolk} we bound the integral on the right hand side of \eqref{3} (up to a global constant) by
$$
\frac{R^n}{R^p} |f(Q(x,2R))|^{p/n} \leq R^{n-p} |Q(y,R')|^{p/n} \leq C(n,p) R^{n-p} (R')^{p} \, .
$$
Hence, by the definition of $\beta$, we obtain
$$
\sum_{j=m}^{\infty} M(j) \leq C_6 \, 2^{m\alpha} \,.
$$
This completes the proof of the lemma.
\end{proof}

We now count the number of critical cubes in the case where $Q(x,R)$ is major. Any such critical cube is one of $2^n$ siblings of a parent major cube. Hence, the number of critical cubes which we will obtain is at most
\begin{equation}\label{2}
\sum_{j=m}^{\infty} 2^n M(j) \leq C_6 \, 2^{m\alpha+n} \le C_7 2^{m\alpha} \, .
\end{equation}
The sub-collection of these critical cubes whose image under $f$ meets $f(E)$ forms a suitable cover of $f(E)$ by sets of diameter at most $r_m'$. By previous comments, its cardinality is at most
$$
C_7 2^{m\alpha} = C_7 \left( \frac{R'}{r_m'} \right)^{\beta},
$$
where we recall that $\beta = p\alpha/(p-n+\alpha)$. It follows that regardless of whether $Q(x,R)$ is minor or major,
$$
N(B(y,R')\cap f(E), r_m') \leq C_7 \left( \frac{R'}{r_m'} \right)^{\beta}
$$ 
which implies
$$
\dim_{A, reg}^{\theta(t)}(f(E)) \le \beta =\frac{p\alpha}{p-n+\alpha}.
$$
This concludes the proof of Theorem \ref{th:main2}. \qed

\subsection{Proof of Theorem \ref{th:main1}}\label{subsec:Assouad-dimension-proof}
We first observe that similar estimates hold for quasiconformal distortion of the quasi-Assouad dimension $\dim_{qA}(E)$. Specifically, if $f:\Omega \to \Omega'$ is a $K$-quasiconformal mapping between domains in $\R^n$, and $E$ is a compact subset of $\Omega$, then
\begin{equation*}\begin{split}
&\left( 1 - \frac{n}{p_O^{\RH}(n,K)} \right) \left( \frac1{\dim_{qA}(E)} - \frac1n \right) \\
& \qquad \le \frac1{\dim_{qA}(f(E))} - \frac1n \le \\
& \left( 1 - \frac{n}{p_O^{\RH}(n,K^{n-1})} \right)^{-1} \left( \frac1{\dim_{qA}(E)} - \frac1n \right) \,.
\end{split}\end{equation*}
This is an immediate consequence of Theorem \ref{th:main2}, obtained by letting $t \rightarrow 0^+ $, which leads to $\theta \nearrow 1$, and using the continuity of $\theta \mapsto \dim_{A,reg}^\theta(E)$ as $\theta \to 1^-$.

The quasi-Assouad dimension $\dim_{qA}(E)$ and the Assouad dimension $\dim_A(E)$ do not agree in general, so Theorem \ref{th:main1} requires an additional argument. Nevertheless, the proof of Theorem \ref{th:main1} is substantially similar to that of Theorem \ref{th:main2} given in the previous subsection.

\begin{proof}[Proof of Theorem \ref{th:main1}]
As in the proof of Theorem \ref{th:main2}, we begin with a series of reductions. Suppose that $\Omega = \Omega' = \R^n$ and that $E$ is unbounded. Since the Assouad dimension of a set is unchanged upon passing to the closure, we may assume without loss of generality that $E$ is closed. If $E = \R^n$ then also $f(E) = \R^n$ and the result is trivial. If $E \subsetneq \R^n$, then choose an open ball in the complement of $E$. Conformal inversion in the boundary of this ball preserves the Assouad dimension of sets, and maps $E$ to a compact set. Moreover, precomposition by a conformal map does not alter the dilatation of the original map $f$. Thus it suffices to assume that $E$ is compact.

We now perform additional reductions as in the beginning of the proof of Theorem \ref{th:main2}, and assume without loss of generality that both $E$ and $f(E)$ are contained in $Q_0$ and that \eqref{eq:Q0} holds. As in the proof of Theorem \ref{th:main2}, it suffices to prove the left hand inequality, as the other inequality follows by considering the inverse map $f^{-1}$. We rewrite the left hand inequality in the form
$$
\dim_{A}(f(E)) \le \beta_0 := \frac{p_O^{\RH}(n,K)\alpha_0}{p_O^{\RH}(n,K)-n+\alpha_0}, \qquad \alpha_0 = \dim_{A}(E).
$$
This follows if we prove
$$
\dim_{A}(f(E)) \le \beta := \frac{p\alpha}{p-n+\alpha}
$$
for all $p<p_O^{\RH}(n,K)$ and $\alpha>\alpha_0$.

Fix such $p$ and $\alpha$, and let $y \in f(E)$ and $R'>0$. As before, applying Corollary \ref{cor:egg-yolk} on the inverse map $f^{-1}$ for the ball $B(y,R')$ yields a ball $B(x,R)$, $x = f^{-1}(y)$ with $f^{-1}(B(y,R'))\subset B(x,R)$. Again, we will cover $f(E) \cap B(y,R')$ by sets of diameter at most $r_m' = 2^{-m\alpha/\beta}R'$ for all $m \ge 0$; this suffices for the desired conclusion. Note that here we allow all nonnegative $m$ in the set of scales, and do not impose an $R$-dependent lower bound on the allowable scales. For $m \ge 0$ we define $r_m := 2^{-m}R$ as before. Since $\alpha>\alpha_0$, for each such $m$ we need at most
\begin{equation}\label{A1}
	C\left(\frac{R}{r_m}\right)^\alpha = C 2^{m\alpha}
\end{equation}
dyadic cubes of side length $2^{-m}R$ to cover $E \cap B(x,R)$. The remainder of the proof follows by Lemma \ref{Le2} in the same fashion as the proof of Theorem \ref{th:main2}.
\end{proof}

\begin{remark}
If $f:\Omega\to \Omega'$ is $K$-quasiconformal and $p<p_I^{\RH}(n,K)$ then $f^{-1}\in W^{1,p}_\loc(\Omega:\R^n)$ and hence satisfies the Morrey-Sobolev inequality with exponent $p$ as in Lemma \ref{Le2}. By following the proof of Theorems \ref{th:main1} and \ref{th:main2} for $f^{-1}$ and $p<p_I^{\RH}(n,K)$, we can prove the inequalities 
	\begin{equation}\label{eq:impmain2}\begin{split}
		\small
		&\left( 1 - \frac{n}{p_O^{\RH}(n,K)} \right) \left( \frac1{\dim_{A,reg}^{\theta(t/K)}(E)} - \frac1n \right) \\
		& \qquad \le \frac1{\dim_{A,reg}^{\theta(t)}(f(E))} - \frac1n \le 
		\\
		&\left( 1 - \frac{n}{p_I^{\RH}(n,K)} \right)^{-1} \left( \frac1{\dim_{A,reg}^{\theta(Kt)}(E)} - \frac1n \right)
	\end{split}\end{equation}and
\begin{equation}\label{eq:impmain1}\small
	\left( 1 - \frac{n}{p_O^{\RH}(n,K)} \right) \left( \frac1{\dim_A E} - \frac1n \right) \le \frac1{\dim_A f(E)} - \frac1n \le \left( 1 - \frac{n}{p_I^{\RH}(n,K)} \right)^{-1} \left( \frac1{\dim_A E} - \frac1n \right)
\end{equation}\normalsize 
respectively. Observe that $p_O^{\RH}(n,K^{n-1})\leq p_I^{\RH}(n,K)$, so \eqref{eq:impmain2} and \eqref{eq:impmain1} are indeed improvements of \eqref{eq:main2} and \eqref{eq:main1} respectively.
\end{remark}




Fraser and Yu \cite{fy:assouad-spectrum} (see also Lemma 3.4.13 in \cite{Fraser2020}) studied the distortion of the Assouad spectrum by bi-H\"older homeomorphisms. Since quasiconformal maps are locally bi-H\"older, it is instructive to consider the relationship between Theorem \ref{th:main2} and the results of \cite{fy:assouad-spectrum}.

Recall that a homeomorphism $f:X\to Y$ between metric spaces is said to be {\it $(\alpha,\beta)$-bi-H\"older}, for $0<\alpha\le 1\le\beta <\infty$, if there exists a constant $C>0$ so that
$$
C^{-1}d(x,x')^\beta \le d(f(x),f(x')) \le C d(x,x')^\alpha \qquad \forall\,x,x' \in X.
$$
If $\alpha < \beta$, it is clear that $X$ must be a bounded space in order for a $(\alpha,\beta)$-bi-H\"older homeomorphism from $X$ to exist. According to \cite[Lemma 3.4.13]{Fraser2020},\footnote{\eqref{eq:biHolder} is stated in \cite{Fraser2020} for unregularized Assouad spectrum values, but easily transfers to the regularized version.} if $X \subset \R^n$ is bounded, $f:X \to \R^n$ is a $(\alpha,\beta)$-bi-H\"older homeomorphism, and $0<\theta<\alpha/\beta$, then
\begin{equation}\label{eq:biHolder}
\frac{1-\beta\theta/\alpha}{\beta(1-\theta)} \dim_{A,reg}^{\beta\theta/\alpha}(X) \le \dim_{A,reg}^\theta(f(X)) \le \frac{1-\alpha\theta/\beta}{\alpha(1-\theta)} \dim_{A,reg}^{\alpha\theta/\beta}(X).
\end{equation}
Now every $K$-quasiconformal map in $\C$ is locally $(\tfrac1K,K)$-bi-H\"older continuous, see e.g.\ \cite[Corollary 3.10.3]{aim:2qc}. It follows from \eqref{eq:biHolder} that if $f:\C\to\C$ is $K$-quasiconformal, $E\subset\C$ is bounded, and $0<\theta<1/K^2$, then
\begin{equation}\label{eq:biHolder2}
\dim_{A,reg}^\theta(f(E)) \le K\, \frac{1-\theta/K^2}{1-\theta} \dim_{A,reg}^{\theta/K^2}(E).
\end{equation}
On the other hand, Corollary \ref{cor:main2} implies that if $t>0$ then
\begin{equation}\label{eq:ours}
\dim_{A,reg}^{\theta(t)}(f(E)) \le K \frac{\dim_{A,reg}^{\theta(t/K)}(E)}{1+\tfrac{K-1}{2}\dim_{A,reg}^{\theta(t/K)}(E)} \,.
\end{equation}
We conclude this section with the following observation, which indicates the range of Assouad spectrum parameters for which \eqref{eq:ours} improves upon \eqref{eq:biHolder2}.

\begin{proposition}\label{prop:final}
Let $d:= \dim_{A,reg}^{\theta(t/K)}(E)$. If $\theta(t)<1/K^2$ and $\theta(t) \le d/2$, then
$$
\frac{d}{1+\tfrac{K-1}{2}d} \le \frac{1-\theta(t)/K^2}{1-\theta(t)} \dim_{A,reg}^{\theta(t)/K^2}(E).
$$
\end{proposition}

The inequality $1/(1+t) \le d/2$ is an implicit bound for $t$, since $d$ is also a function of $t$. However, if $\ovdimB(E)>0$ then it suffices to assume $\theta(t) \le \min\{1/K^2,\ovdimB(E)/2\}$.

\begin{proof}[Proof of Proposition \ref{prop:final}]
Let $\theta_1 = \theta(t)/K^2$ and $\theta_2 = \theta(t/K)$. Then $d = \dim_{A,reg}^{\theta_2}(E)$ and the conclusion reads
\begin{equation}\label{eq:prop-conclusion}
\frac{d}{1+\tfrac{K-1}{2}d} \le \frac{1-\theta_1}{1-\theta(t)} \dim_{A,reg}^{\theta_1}(E).
\end{equation}
Since $\theta_1 \le \theta_2$, Theorem 3.3.1 of \cite{Fraser2020} implies that
$$
\dim_{A,reg}^{\theta_1}(E) \ge \frac{1-\theta_2}{1-\theta_1} \dim_{A,reg}^{\theta_2}(E)
$$
and so \eqref{eq:prop-conclusion} is implied by
\begin{equation}\label{eq:prop-conclusion2}
\frac{d}{1+\tfrac{K-1}{2}d} \le \frac{1-\theta_2}{1-\theta(t)}d.
\end{equation}
Since $1-\theta_2=\tfrac{t}{K+t}$ and $1-\theta(t)=\tfrac{t}{1+t}$, \eqref{eq:prop-conclusion2} reads $1 / (1+\tfrac{K-1}{2}d) \le (1+t)/(K+t)$, 
which is equivalent to the assumption $\theta(t) \le d/2$.
\end{proof}

\begin{remark}
Recall from the introduction the inequalities
$$
p_O^{\RH}(n,K) \le p_O^{\Sob}(n,K) \le \frac{nK}{K-1}
$$
valid for all $n \ge 3$ and $K \ge 1$. Iwaniec and Martin \cite{im:gft} have shown that for any $n \ge 3$ there exists a constant $\lambda = \lambda(n) \ge 1$ so that
\begin{equation}\label{eq:pnK-lower-bound}
p_O^{\Sob}(n,K) \ge \frac{n\lambda K}{\lambda K - 1}.
\end{equation}
The value of $\lambda(n)$ obtained in \cite{im:gft} is the smallest possible constant $\lambda \ge 1$ which makes the inequality
\begin{equation}\label{eq:IMconj}
\left| \int |Df|^{p-n} \, (\det Df) \right| \le \lambda\left|1-\frac{n}{p}\right| \int |Df|^p
\end{equation}
valid for distributional maps $f:\R^n \to \R^n$ with $L^p$ differential $Df:\R^n \to \R^{n\times n}$, and a stronger conjecture (also due to Iwaniec and Martin) is that \eqref{eq:IMconj} holds with $\lambda=1$ for all such maps $f$. 

In fact, a closer analysis of the proof of \cite[Theorem 14.4.1]{im:gft} reveals that
\begin{equation}\label{eq:pnK-lower-bound2}
p_O^{\RH}(n,K) \ge \frac{n\lambda K}{\lambda K - 1}.
\end{equation}
A more precise version of \eqref{eq:gv2} follows. Theorem 17.4.1 of \cite{im:gft} asserts that if $f:\Omega \to \Omega'$ is quasiconformal between domains in $\R^n$ and $E \subset \Omega$ is closed subset, then
\begin{equation}\label{eq:lambda-estimates}
\frac1{K_O(f)\lambda(n)} \left( \frac1{\dim_H(E)}-\frac1n \right) \le \frac1{\dim_H(f(E))} - \frac1n \le \lambda(n) {K_I(f)} \left( \frac1{\dim_H(E)} - \frac1n \right) \, .
\end{equation} 
\end{remark}
We obtain analogous statements for the Assouad dimension and the Assouad spectrum.

\begin{corollary}\label{cor:main1}
The estimates for quasiconformal distortion of the Assouad dimension and spectrum can be sharpened to match those in \eqref{eq:lambda-estimates}. Specifically, the coefficients ${1-n/{p_O^{\RH}(n,K)}}$ and $(1-n/{p_I^{\RH}(n,K)})^{-1}$ can be replaced by $(K_O(f)\lambda(n))^{-1}$ and $\lambda(n) K_I(f)$ respectively, in inequalities \eqref{eq:impmain1} and \eqref{eq:impmain2}.
\end{corollary}

\begin{proof}
The left hand inequality follows immediately from \eqref{eq:pnK-lower-bound2}. The right hand inequality follows by applying the left hand inequality to the inverse of $f$.
\end{proof}

\section{Concluding remarks and open questions}\label{sec:questions}

\begin{remark}
The upper bound for Assouad dimension in Theorem \ref{th:main1} has the same form as in the analogous theorem for Hausdorff dimension. The upper bound depends only on the dimension of the source set $E$, on $n$, and on the optimal exponent of higher Sobolev regularity. This naturally leads us to pose the following question.
\end{remark}

\begin{question}\label{q:1}
What can be said about upper bounds for distortion of Assouad dimension or the Assouad spectrum under (not necessarily quasiconformal) maps $f \in W^{1,p}(\R^n:\R^N)$, $p>n$?
\end{question}

The analogous question for Hausdorff and box-counting dimensions was considered by Kaufman \cite{Kaufman} for sets $E \subset \R^n$, and by Balogh, Monti and the second author in \cite{bmt:frequency} for generic elements in parameterized families of subsets of $\R^n$.

In particular, Question \ref{q:1} remains not completely resolved even if we {\bf do} assume that $f$ is quasiconformal. Of course, in that situation the estimates in \eqref{eq:main1} and \eqref{eq:main2} provide upper (resp.\ lower) bounds for the dimension of $f(E)$ in terms of the dimension $\alpha_0$ of $E$ and the optimal Sobolev regularity exponents $p_O(n,K)$ (resp.\ $p_I(n,K)$). However, what is still not known is whether better estimates hold if $f$ has greater Sobolev regularity than that dictated by the universal exponents $p_O(n,K)$ and $p_I(n,K)$ associated to the dilatation of $f$. One may ask whether the usual upper bound $p\alpha_0/(p-n+\alpha_0)$ is still valid if $f$ is merely assumed to be a quasiconformal map in the local Sobolev space $W^{1,p}_\loc$. In particular, if $f:\R^n \to \R^n$ is quasiconformal and Lipschitz, must it be the case that $\dim_A f(E) \le \dim_A(E)$ for all sets $E \subset \R^n$? Recall (as previously observed) that Lipschitz mappings can in general raise the Assouad dimension of sets.

Another natural question which arises is the following.

\begin{question}
Give quantitative upper and lower bounds for distortion of Assouad dimension under quasisymmetric maps of $\R$.
\end{question}

It is known that if $f:\R\to\R$ is $\eta$-quasisymmetric and $E \subset \R$ has $\dim_A(E) = \alpha \in (0,1)$, then $\dim_A(f(E)) \le \beta = \beta(\alpha,\eta)<1$. This follows from the quantitative equivalence of porosity with non-full Assouad dimension (valid in any $\R^n$) \cite[Theorem 5.2]{Luukkainen} and the quantitative quasisymmetric invariance of porosity \cite{vai:porosity}. The exact formula for $\beta(\alpha,\eta)$ stemming from this argument is complicated and unlikely to be sharp. It is well-known that the analogous statement for Hausdorff dimension is false. Indeed, there exist quasisymmetric maps $f:\R\to\R$ and subsets $E \subset \R$ for which both $E$ and $\R \setminus f(E)$ have Hausdorff dimension as small as we please. See, for example, \cite{tuk:qcgroup}.

\begin{remark}
Another consequence of Theorem \ref{th:main1} is that Assouad dimension of compact sets is invariant under conformal mappings. In dimensions three and higher, this statement provides no new information, due to Liouville's theorem and the M\"obius invariance of Assouad dimension \cite[Theorem A.10]{Luukkainen}. However, it is a new result for planar conformal maps.
\end{remark}

\begin{example}
The restriction to compact sets in the previous remark is necessary. Let $\Omega$ be the set of $z \in \C$ so that $\Real(z)>0$ and $|\Imag(z)|<\pi$. Let $f(z) = \exp(-z)$. Let $E = \N \subset \Omega$. Then $\dim_A(E) = 1$, but $\dim_A(f(E)) = 0$.
\end{example}

Dilatation-independent results for these notions of dimension may also be of interest. The {\it global quasiconformal dimension} of a set $E \subset \R^n$ is the infimum of dimensions of images $f(E)$, where the infimum is taken over all quasiconformal self-maps of $\R^n$. For values $\alpha \in [1,n)$ there exist Ahlfors $\alpha$-regular sets $E \subset \R^n$ which are minimal for quasiconformal dimension distortion. Note that all notions of dimension considered in this paper (Hausdorff, box-counting, Assouad, and the Assouad spectrum) agree for an Ahlfors regular set. Thus for any $0<\theta<1$ and any $1\le \alpha <n$ there exists a set $E \subset \R^n$ with $\dim_A^\theta(E) = \alpha$ and which is minimal for global quasiconformal Assouad spectrum dimension with parameter $\theta$. It is known that sets of Hausdorff (respectively, Assouad) dimension strictly less than one have global quasiconformal Hausdorff (respectively, Assouad) dimension zero; these results can be found in \cite{kov:confdim} and (respectively) \cite{tys:assouad}.

\begin{conjecture}\label{conj:gqcd}
Let $0\le \theta < 1$ and let $E \subset \R^n$ satisfy $\dim_A^\theta(E)<1$. Then the global quasiconformal Assouad spectrum dimension of $E$ with parameter $\theta$ is equal to zero. Here we interpret $\dim_A^0$ to be the upper box-counting dimension $\ovdimB$.
\end{conjecture}

To prove Conjecture \ref{conj:gqcd} it suffices to establish the case $\theta=0$, i.e., to prove the result for upper box-counting dimension. This follows from known estimates for Assouad spectrum values; see Proposition \ref{prop:Assouad-spectrum-facts}(5). Assume that Conjecture \ref{conj:gqcd} has been established for $\theta=0$. Let $E$ and $0<\theta<1$ be such that $\dim_A^\theta(E)<1$. Then $\ovdimB(E)<1$ and hence there exist quasiconformal maps $f$ of $\R^n$ for which $\ovdimB(f(E))$ is arbitrarily small. Inequality \eqref{eq:fraser-bound} then implies that $\dim_A^\theta(f(E))$ can also be made arbitrarily small by varying over all quasiconformal self-maps $f$ of $\R^n$.

\bibliographystyle{acm}
\bibliography{QC-distortion-Assouad-dimension-revision}
\end{document}